\documentclass[11pt]{article}
\usepackage[a4paper,centering,vscale=0.68,hscale=0.65]{geometry}
\usepackage{amsfonts,amssymb,amsmath,amsthm}
\usepackage{mathrsfs}
\usepackage{hyperref}

\newtheorem{theorem}{Theorem}[section]
\newtheorem{lemma}[theorem]{Lemma}
\newtheorem{proposition}[theorem]{Proposition}
\newtheorem{definition}[theorem]{Definition}

\newcommand{\E}{\mathop{{}\mathbb{E}}}
\newcommand{\cL}{\mathscr{L}}
\newcommand{\cF}{\mathscr{F}}
\renewcommand{\P}{\mathbb{P}}
\newcommand{\erre}{\mathbb{R}}

\renewcommand{\div}{\mathop{\mathrm{div}}}
\newcommand{\norm}[1]{\lVert{#1}\rVert}

\title{On the well-posedness of SPDEs with singular drift in
  divergence form} 

\author{
Carlo Marinelli and Luca Scarpa\\
{\footnotesize \emph{Department of Mathematics, University College London, United Kingdom}}
}

\date{\normalsize January 28, 2017}

\begin{document}
\maketitle

\begin{abstract}
  We prove existence and uniqueness of strong solutions for a class of
  second-order stochastic PDEs with multiplicative Wiener noise and
  drift of the form $\div \gamma(\nabla \cdot)$, where $\gamma$ is a
  maximal monotone graph in $\erre^n \times \erre^n$ obtained as the
  subdifferential of a convex function satisfying very mild
  assumptions on its behavior at infinity. The well-posedness result
  complements the corresponding one in our recent work
  \href{http://arxiv.org/abs/1612.08260}{arXiv:1612.08260} where,
  under the additional assumption that $\gamma$ is single-valued, a
  solution with better integrability and regularity properties is
  constructed. The proof given here, however, is self-contained.
\end{abstract}

\section{Introduction and main result}
Let us consider the stochastic partial differential equation
\begin{equation}
  \label{eq:0}
  du(t) - \div\gamma(\nabla u(t))\,dt \ni B(t, u(t))\,dW(t),
  \qquad u(0)=u_0,
\end{equation}
posed on $L^2(D)$, with $D$ a bounded domain of $\erre^n$ with smooth
boundary. The following assumptions will be in force: (a)
$\gamma$ is the subdifferential of a
lower semicontinuous convex function $k:\erre^n \to \erre_+$ with
$k(0)=0$ and such that
\[
\lim_{|x| \to \infty} \frac{k(x)}{|x|} = + \infty,
\qquad
\limsup_{|x|\to\infty} \frac{k(-x)}{k(x)} < +\infty
\]
(in particular, $\gamma$ is a maximal monotone graph in
$\erre^n \times \erre^n$ whose domain coincides with $\erre^n$); (b)
$W$ is a cylindrical Wiener process on a separable Hilbert space $H$,
supported by a filtered probability space
$(\Omega,\cF,(\cF_t)_{t\in[0,T]},\P)$ satisfying the ``usual
conditions''; (c) $B$ is a map from $\Omega \times [0,T]\times L^2(D)$
to $\cL^2(H,L^2(D))$, the space of Hilbert-Schmidt operators from $H$
to $L^2(D)$, that is Lipschitz-continuous and has linear growth with
respect to its third argument, uniformly with respect to the other
two, and is such that $B(\cdot,\cdot,a)$ is measurable and adapted for
all $a \in L^2(D)$.

Under the additional assumption that $\gamma$ is a (single-valued)
continuous function, we proved in \cite{cm:div} that \eqref{eq:0}
admits a strong solution $u$, which is unique within a set of
processes satisfying mild integrability conditions. The solution of
\cite{cm:div} is constructed pathwise, i.e. for each
$\omega\in\Omega$, so that, as is natural to expect, measurability
problems arise with respect to the usual $\sigma$-algebras on $\Omega
\times [0,T]$ used in the theory of stochastic processes. Precisely
because of such an issue we needed to assume $\gamma$ to be
single-valued.

The purpose of this note is to provide an alternative approach to
establish the well-posedness of \eqref{eq:0} that, avoiding pathwise
constructions, is simpler than that of \cite{cm:div} and does not need
any extra assumption on $\gamma$. The price to pay is that the
solution we obtain here is less regular than that of \cite{cm:div}.
We also refer to \cite{luca} for a related result obtained by
analogous methods.

\medskip

Let us define the concept of solution to \eqref{eq:0} we shall be
working with.
\begin{definition}
\label{def:strong}
  Let $u_0$ be an $L^2(D)$-valued $\cF_0$-measurable random variable. A
  \emph{strong solution} to equation \eqref{eq:0} is a couple
  $(u,\eta)$ satisfying the following properties:
  \begin{itemize}
  \item[(i)] $u$ is a measurable and adapted $L^2(D)$-valued
    process such that
    \[
    u \in L^1(0,T;W^{1,1}_0(D)) 
    \qquad \text{and} \qquad
    B(\cdot,u) \in L^2(0,T;\cL^2(U,L^2(D))) \quad \P\text{-a.s.};
    \]
  \item[(ii)] $\eta$ is a measurable and adapted $L^1(D)^n$-valued process
    such that
    \[
    \eta \in L^1(0,T;L^1(D)^n) \quad\P\text{-a.s.}\,, \qquad
    \eta\in\gamma(\nabla u) \quad\text{a.e.~in } \Omega\times(0,T)\times D;
    \]
  \item[(iii)] one has, as an equality in $L^2(D)$,
   \begin{equation} \label{eq:0i}
   u(t)-\int_0^t{\div\eta(s)\,ds} = u_0
   +\int_0^t B(s,u(s))\,dW(s) \quad \P\text{-a.s.} \quad \forall\, t\in[0,T].
    \end{equation}
  \end{itemize}
\end{definition}
Note that \eqref{eq:0i} has to be intended in the sense of
distributions.  In particular, since $\eta\in L^1(D)^n$, the integrand
in the second term of \eqref{eq:0i} does not, in general, take values
in $L^2(D)$. However, the conditions on $B$ imply that the stochastic
integral in \eqref{eq:0i} is an $L^2(D)$-valued local martingale,
hence the term involving the divergence of $\eta$ turns out to be
$L^2(D)$-valued by comparison.

\medskip

We can now state our main result. Here and in the following
$k^*:\erre^n \to \erre_+$ is the convex conjugate of $k$, defined
as $k^*(y) := \sup_{x\in\erre^n} \bigl( x\cdot y-k(x) \bigr)$.
\begin{theorem}
  \label{thm:main}
  Let $u_0 \in L^2(\Omega,\cF_0;L^2(D))$. Then equation \eqref{eq:0}
  admits a unique strong solution $(u,\eta$) such that
  \begin{align*}
  &\sup_{t\leq T}\, \E\norm{u(t)}_{L^2(D)}^2 
  + \E\int_0^T \norm{u(t)}_{W^{1,1}_0(D)}\,dt < \infty,\\
  &\E\int_0^T\norm{\eta(t)}_{L^1(D)^n}\,dt <\infty,\\
  &\E\int_0^T \bigl( \norm{k(\nabla u(t))}_{L^1(D)} 
              + \norm{k^*(\eta(t))}_{L^1(D)} \bigr)\,dt < \infty.
  \end{align*}
  Moreover, the solution map $u_0 \mapsto u$ is Lipschitz-continuous
  from $L^2(\Omega;L^2(D))$ to $L^\infty(0,T;L^2(\Omega;L^2(D)))$, and
  $u$ is weakly continuous as a function on $[0,T]$ with values in
  $L^2(\Omega;L^2(D))$.
\end{theorem}

Under the extra assumption of $\gamma$ being single-valued, the
solution obtained in~\cite{cm:div} is more regular in the sense that
$\E \sup_{t\leq T} \norm{u(t)}^2_{L^2(D)}$ is finite, the solution
map is Lipschitz-continuous from $L^2(\Omega;L^2(D))$ to
$L^2(\Omega;L^\infty(0,T;L^2(D)))$, and $u(\omega,\cdot)$ is weakly
continuous as a function on $[0,T]$ with values in $L^2(D)$ for
$\P$-a.a. $\omega \in \Omega$.

\medskip

\noindent\textbf{Acknowledgements.}  The authors are partially supported by The
Royal Society through its International Exchange Scheme. Parts of this
chapter were written while the first-named author was visiting the
Interdisziplin\"ares Zentrum f\"ur Komplexe Systeme at the University
of Bonn, hosted by Prof.~S.~Albeverio.

\section{Well-posedness of an auxiliary equation}
\label{sec:add}
The goal of this section is to prove well-posedness of a version of
\eqref{eq:0} with additive noise. Namely, we consider the initial
value problem
\begin{equation}
\label{eq:add}
du(t) - \div \gamma(\nabla u(t))\,dt \ni G(t)\,dW(t),
\qquad u(0)=u_0,
\end{equation}
where $G \in L^2(\Omega \times [0,T];\cL^2(H,L^2(D)))$ is a measurable and
adapted process.
\begin{proposition}   \label{prop:add}
  Equation \eqref{eq:add} admits a unique strong solution $(u,\eta)$
  satisfying the same integrability and weak continuity conditions of
  Theorem~\ref{thm:main}.
\end{proposition}

We introduce the regularized equation
\[
du_\lambda(t) - \div \gamma_\lambda (\nabla u_\lambda(t))\,dt -
\lambda\Delta u_\lambda(t)\,dt = G(t)\,dW(t), \qquad
u_\lambda(0)=u_0,
\]
where $\gamma_\lambda: \erre^n \to \erre^n$, $\gamma_\lambda :=
\frac{1}{\lambda} (I - (I+\lambda\gamma)^{-1})$, for any $\lambda>0$,
is the Yosida approximation of $\gamma$, and $\Delta:H^1_0(D) \to
H^{-1}(D)$ is the (variational) Dirichlet Laplacian.
Since $\gamma_\lambda$ is monotone and Lipschitz-continuous, the
classical variational approach (see~\cite{KR-spde,Pard} as well as
\cite{LiuRo}) yields the existence of a unique predictable process
$u_\lambda$ with values in $H^1_0(D)$ such that
\[
\E \norm{u_\lambda}^2_{C([0,T];L^2(D))} 
+ \E\int_0^T \norm{u_\lambda(t)}^2_{H^1_0(D)}\,dt < \infty
\]
and
\begin{equation}
  \label{eq:reg}
  u_\lambda(t) - \int_0^t \div\gamma_\lambda(\nabla u_\lambda(s))\,ds
  - \lambda\int_0^t \Delta u_\lambda(s)\,ds = u_0 + \int_0^t G(s)\,dW(s)  
\end{equation}
$\P$-a.s. in $H^{-1}(D)$ for all $t \in [0,T]$.

We are now going to prove a priori estimates and weak compactness in
suitable topologies for $u_\lambda$ and related processes. These will
allow us to pass to the limit as $\lambda \to 0$ in
\eqref{eq:reg}. 

For notational parsimony, we shall often write, for any $p \geq 0$,
$L^p_\omega$, $L^p_t$, and $L^p_x$ in place of $L^p(\Omega)$,
$L^p(0,T)$, and $L^p(D)$, respectively, and $C_t$ to denote
$C([0,T])$. Other similar abbreviations are self-explanatory. The
$L^2(D)$-norm will be denoted simply by $\norm{\cdot}$. If a function
$f:D \to \erre^n$ is such that each component $f^j$, $j=1,\ldots,n$,
belongs to $L^p(D)$, we shall just write $f \in L^p(D)$ rather than $f
\in L^p(D)^n$. The notation $a \lesssim b$ means that $a \leq Nb$ for
a positive constant $N$.

\begin{lemma}  \label{lm:aspetta}
  There exists a constant $N$ such that
  \begin{align*}
    &\norm{u_\lambda}_{L^2_\omega C_t L^2_x} + 
    \lambda^{1/2} \norm{\nabla u_\lambda}_{L^2_{t,\omega,x}}
    + \norm{\gamma_\lambda(\nabla u_\lambda)%
      \cdot \nabla u_\lambda}_{L^1_{t,\omega,x,}}\\
    &\hspace{3em} < N \bigl( \norm{u_0}_{L^2_{\omega,x}} 
      + \norm{G}_{L^2_{t,\omega} \cL^2(H,L^2_x)} \bigr).
  \end{align*}
\end{lemma}
\begin{proof}
  It\^o's formula for the square of the norm in $L^2_x$ yields
  \begin{align*}
  &\norm{u_\lambda(t)}^2 
  + 2\int_0^t\!\!\int_D \gamma(\nabla u_\lambda(s)) \cdot 
    \nabla u_\lambda(s)\,dx\,ds
  + 2\lambda \int_0^t \norm{\nabla u_\lambda(s)}^2\,ds\\
  &\hspace{3em} = \norm{u_0}^2
  + 2\int_0^t u_\lambda(s) G(s)\,dW(s) 
  + \int_0^t \norm{G(s)}^2_{\cL^2(H,L^2_x)}\,ds,
  \end{align*}
  hence, taking the supremum in time and expectation,
  \begin{align*}
  &\E\norm{u_\lambda}^2_{C_t L^2_x} 
  + \E\int_0^T\!\!\int_D \gamma_\lambda(\nabla u_\lambda(s))%
     \cdot \nabla u_\lambda(s)\,dx\,ds
  + \lambda\E\norm{\nabla u_\lambda}^2_{L^2_{t,x}}\\
  &\hspace{5em} \lesssim
  \E\norm{u_0}^2 + \E\norm{G}^2_{L^2_t \cL^2(H,L^2_x)}
  + \E\sup_{t\in[0,T]} \left|{\int_0^t u_\lambda(s) G(s)\,dW(s)}\right|,
  \end{align*}
  where, by Davis' inequality (see,~e.g.,~\cite{cm:Expo16}), the ideal
  property of Hilbert-Schmidt operators
  (see,~e.g.,~\cite[p.~V.52]{Bbk:EVT}), and the elementary inequality
  $ab \leq \varepsilon a^2 + b^2/\varepsilon$ $\forall a,b\geq 0$,
  $\varepsilon>0$,
  \begin{align*}
  \E\sup_{t\in[0,T]} \left\vert 
      \int_0^t u_\lambda(s) G(s)\,dW(s) \right\vert &\lesssim
  \E\biggl( \int_0^T \norm{u_\lambda(s)G(s)}^2_{\cL^2(H,\erre)}\,ds\biggr)^{1/2}\\
  &\leq \varepsilon \E \norm{u_\lambda}^2_{C_t L^2_x}
  + N(\varepsilon) \E\int_0^T \norm{G(s)}^2_{\cL^2(H,L^2_x)}\,ds
  \end{align*}
  for any $\varepsilon>0$. To conclude it suffices to choose
  $\varepsilon$ small enough.
\end{proof}

\begin{lemma}
  \label{lm:L1a}
  The families $(\nabla u_\lambda)$ and
  $(\gamma_\lambda(\nabla u_\lambda))$ are relatively weakly compact
  in $L^1(\Omega \times (0,T) \times D)$.
\end{lemma}
\begin{proof}
  Recall that, for any $y$, $r \in \erre^n$, ones has
  $k(y)+k^*(r)=r\cdot y$ if and only if $r \in \partial
  k(y)=\gamma(y)$. Therefore, since
  \[
  \gamma_\lambda(x) \in \partial k\bigl((I+\lambda\gamma)^{-1}x\bigr)
  = \gamma\bigl((I+\lambda\gamma)^{-1}x\bigr) \qquad \forall x \in \erre^n,
  \]
  we deduce by the definition of $\gamma_\lambda$ that
  \begin{align}
  \nonumber
  k\bigl((I+\lambda\gamma)^{-1}x\bigr) + k^*\bigl(\gamma_\lambda(x)\bigr)
  &= \gamma_\lambda(x)\cdot(I+\lambda\gamma)^{-1}x \\
  \label{eq:pluto}
  &=\gamma_\lambda(x)\cdot x-\lambda\left|{\gamma_\lambda(x)}\right|^2
  \leq \gamma_\lambda(x)\cdot x
  \qquad \forall x \in \erre^n\,.
  \end{align}
  (See,~e.g.,~\cite{lema} for all necessary facts from convex analysis
  used in this note.) Hence, taking Lemma~\ref{lm:aspetta} into account,
  there exists a constant $N>0$, independent of $\lambda$, such that
  \[
  \E\int_0^T\!\!\int_D k^*\bigl(\gamma_\lambda(\nabla u_\lambda)\bigr)
  \leq \E\int_0^T\!\!\int_D \gamma_\lambda(\nabla u_\lambda)\cdot \nabla u_\lambda
  < N.
  \]
  The assumptions on $k$ imply that its convex conjugate $k^*$ is also
  convex, lower semicontinuous and such that $\lim_{\lvert y \rvert
    \to \infty} k^*(y)/\lvert y \rvert = +\infty$. Therefore a simple
  modification of the criterion by de la Vall\'ee Poussin implies that
  $(\gamma_\lambda(\nabla u_\lambda))$ is uniformly integrable on
  $\Omega \times (0,T) \times D$, hence that it is relatively weakly
  compact in $L^1_{t,\omega,x}$ by the Dunford-Pettis
  theorem. A completely analogous argument shows that
  \begin{align*}
  \E\int_0^T\!\!\int_D k\bigl((I+\lambda\gamma)^{-1}\nabla u_\lambda\bigr)
  \leq \E\int_0^T\!\!\int_D \gamma_\lambda(\nabla u_\lambda) \cdot%
    \nabla u_\lambda < N,
  \end{align*}
  hence that $(I+\lambda\gamma)^{-1}\nabla u_\lambda$ is relatively
  weakly compact in $L^1_{t,\omega,x}$. Moreover,
  since
  $\nabla u_\lambda=(I+\lambda\gamma)^{-1}\nabla
  u_\lambda+\lambda\gamma_\lambda(\nabla u_\lambda)$, it also follows
  that $(\nabla u_\lambda)$ is relatively weakly compact in
  $L^1_{t,\omega,x}$.
\end{proof}

Thanks to Lemmata \ref{lm:aspetta} and \ref{lm:L1a}, there exists a
subsequence of $\lambda$, denoted by the same symbol, and processes $u
\in L^\infty_t L^2_{\omega,x} \cap L^1_{t,\omega} W^{1,1}_0$ and $\eta
\in L^1_{t,\omega,x}$ such that
\begin{align*}
    u_\lambda &\longrightarrow u & &\text{weakly* in } L^\infty_t L^2_{\omega,x},\\
    u_\lambda &\longrightarrow u & &\text{weakly in } L^1_{t,\omega} W^{1,1}_0,\\
    \gamma_\lambda(\nabla u_\lambda) &\longrightarrow \eta
    & &\text{weakly in } L^1_{t,\omega,x},\\
    \lambda u_\lambda &\longrightarrow 0 & &\text{weakly in } L^2_{t,\omega}H^1_0.
\end{align*}
as $\lambda \to 0$.
Let $t \in [0,T]$ be arbitrary but fixed. The fourth convergence above
implies
\[
  \lambda\int_0^t\Delta u_\lambda(s)\,ds \longrightarrow 0
  \qquad \text{ in } L^2_\omega H^{-1},
\]
while the third yields, for any $\varphi \in L^\infty_\omega W^{1,\infty}$,
\[
\E\int_0^t\!\!\int_D \gamma_\lambda(\nabla u_\lambda(s)) \cdot%
   \nabla\varphi\,dx\,ds \longrightarrow
\E\int_0^t\!\!\int_D\eta(s) \cdot \nabla\varphi\,dx\,ds,
\]
hence
$\displaystyle \E\int_0^t \langle \div\gamma_\lambda(\nabla
u_\lambda(s)), \varphi\rangle\,ds \longrightarrow \E\int_0^t\langle
\div\eta(s), \varphi\rangle \,ds$. Therefore, recalling
\eqref{eq:reg}, by difference we deduce that
\[
\E \langle u_\lambda(t),\varphi \rangle \longrightarrow 
\E \langle u(t),\varphi \rangle.
\]
Consequently, since $u_\lambda(t)$ is bounded in $L^2_\omega L^2_x$,
we also have that $u_\lambda(t)\rightarrow u(t)$ weakly in
$L^2_\omega L^2_x$.  Taking the limit as $\lambda \to 0$ in
\eqref{eq:reg} thus yields
\[
u(t) - \int_0^t\div\eta(s)\,ds
= u_0 + \int_0^t G(s)\,dW(s) \qquad\text{in } L^1_\omega V_0',
\]
where $V_0'$ is the (topological) dual of a separable Hilbert space
$V_0$ embedded continuously and densely in $H^1_0$, and continuously
in $W^{1,\infty}$.  The identity immediately implies that
$u \in C_t L^1_\omega V_0'$. Since
$u \in L^\infty_t L^2_\omega L^2_x$, it follows by a result of Strauss
(see~\cite[Theorem~2.1]{Strauss}) that $u$ is a weakly continuous
function on $[0,T]$ with values in $L^2_\omega L^2_x$.

By Mazur's lemma there exist sequences of convex combinations of
$\gamma_\lambda(\nabla u_\lambda)$ that converge $\eta$ in (the norm
topology of) $L^1_{t,\omega,x}$, thus also, passing to a subsequence,
$\P \otimes dt$-almost everywhere in $L^1_x$. Similarly, since
$u_\lambda \to u$ weakly* in $L^\infty_tL^2_{\omega,x}$ implies that
$u_\lambda \to u$ weakly in $L^2_{t,\omega,x}$, there exist sequences
of convex combinations of $u_\lambda$ that converge to $u$ $\P \otimes
dt$-almost everywhere in $L^2_x$. Since convex combinations of
$(u_\lambda)$ and of $(\gamma_\lambda(\nabla u_\lambda))$ are (at
least) predictable and adapted, respectively, it follows that $u$ is
predictable and $\eta$ is measurable and adapted.
Moreover, thanks to the weak lower semicontinuity of convex
integrals, one has
\[
\E\int_0^T\!\!\int_D \bigl(k(\nabla u)+k^*(\eta)\bigr) < \infty.
\]
In order to show that $\eta \in \gamma(\nabla u)$ for
a.a. $(\omega,t,x)$, we need the following ``energy identity''.
\begin{lemma}\label{lm:ito}
  Assume that
  \[
  y(t) + \alpha \int_0^t y(s)\,ds - \int_0^t \div \zeta(s)\,ds =
  y_0 + \int_0^t C(s)\,dW(s)
  \]
  in $L^2_x$ $\P$-a.s. for all $t \in [0,T]$, where
  $\alpha \in \erre$, $y_0 \in L^2_{\omega,x}$ is $\cF_0$-measurable, and
  \[
  y \in L^\infty_tL^2_{\omega,x}\cap L^1_{t,\omega}W^{1,1}_0, \qquad
  \zeta\in L^1_{t,\omega,x}, \qquad C \in L^2_{t,\omega}\cL^2(H,L^2_x)
  \]
  are measurable and adapted processes such that
  $k(c\nabla y)+k^*(c\zeta)\in L^1_{t,\omega,x}$ for a constant
  $c>0$. Then
  \begin{align*}
  &\E\norm{y(t)}^2 + 2\alpha \E \int_0^t\norm{y(s)}^2\,ds
  + 2\E\int_0^t\!\!\int_D \zeta\cdot\nabla y\,dx\,ds\\
  &\hspace{3em}
  = \E\norm{y_0}^2 + \E \int_0^t \norm{C(s)}^2_{\cL^2(H,L^2_x)}\,ds
  \qquad \forall t \in [0,T].
  \end{align*}
\end{lemma}
\begin{proof}
  Let $m\in\mathbb{N}$ be such that such that $(I-\delta\Delta)^{-m}$
  maps $L^1_x$ into $H^1_0\cap W^{1,\infty}$, and use the notation
  $h^\delta:=(I-\delta\Delta)^{-m}h$ for any $h$ taking values in
  $L^1_x$. One has
  \begin{equation}
  \label{eq:Ito-reg}
  y^\delta(t)+\alpha\int_0^ty^\delta(s)\,ds-\int_0^t\div\zeta^\delta(s)\,ds=
  y^\delta_0+\int_0^tC^\delta(s)\,dW(s)
  \end{equation}
  $\P$-a.s. for all $t\in[0,T]$, as an equality in $L^2_x$, for which
  It\^o's formula yields
  \begin{align*}
  &\norm{y^\delta(t)}^2 + 2\alpha \int_0^t \norm{y^\delta(s)}^2\,ds
  + 2\int_0^t\!\!\int_D \zeta^\delta \cdot \nabla y^\delta\,dx\,ds\\
  &\hspace{3em} = \norm{y^\delta_0}^2
  + \int_0^t \norm{C^\delta(s)}^2_{\cL^2(H,L^2_x)}\,ds
  + \int_0^t y^\delta(s) C^\delta(s)dW(s).
  \end{align*}
  It is evident from \eqref{eq:Ito-reg} that $y^\delta$ is a
  continuous $L^2_x$-valued process, hence the stochastic integral
  $(y^\delta C^\delta) \cdot W$ on the right-hand side of the above
  identity is a continuous local martingale. Let $(T_n)$ be a
  localizing sequence, and multiply the previous identity by
  $1_{[\!\![0,T_n]\!\!]}$, to obtain, thanks to $\E(y^\delta
  C^\delta) \cdot W(\cdot \wedge T_n)=0$,
  \begin{align*}
  &\E \norm{y^\delta(t \wedge T_n)}^2 
  + 2\alpha \E\int_0^{t\wedge T_n} \norm{y^\delta(s)}^2\,ds
  + 2\E\int_0^{t \wedge T_n}\!\!\int_D \zeta^\delta \cdot \nabla y^\delta\,dx\,ds\\
  &\hspace{3em} = \E\norm{y^\delta_0}^2
  + \E\int_0^{t \wedge T_n} \norm{C^\delta(s)}^2_{\cL^2(H,L^2_x)}\,ds.
  \end{align*}
  Letting $n$ tend to $\infty$, the dominated convergence theorem
  yields
  \begin{align*}
  &\E\norm{y^\delta(t)}^2
  + 2\alpha \E\int_0^t \norm{y^\delta(s)}^2\,ds
  + 2\E\int_0^t\!\!\int_D\zeta^\delta\cdot\nabla y^\delta\,dx\,ds\\
  &\hspace{3em} = \E\norm{y^\delta_0}^2
  + \E\int_0^t\norm{C^\delta(s)}^2_{\cL^2(H,L^2_x)}\,ds
  \end{align*}
  for all $t\in[0,T]$. We are now going to pass to the limit as
  $\delta \to 0$: the first and second terms on the left-hand side and
  the first on the right-hand side clearly converge to
  $\E\norm{y(t)}^2$, $2\alpha \E\int_0^t\norm{y(s)}^2\,ds$ and
  $\E\norm{y_0}^2$, respectively. Properties of Hilbert-Schmidt
  operators and the dominated convergence theorem also yield
  \[
  \lim_{\delta \to 0} \E\int_0^t \norm{C^{\delta}(s)}_{\cL^2(H,L^2_x)}^2\,ds
  = \E\int_0^t \norm{C(s)}_{\cL^2(H,L^2_x)}^2\,ds
  \]
  for all $t \in [0,T]$.
  To conclude it then suffices to show that $\nabla y^{\delta} \cdot
  \zeta^{\delta} \to \nabla y\cdot \zeta$ in $L^1_{t,\omega,x}$.
  Since $\nabla y^{\delta} \to \nabla y$ and $\zeta^{\delta} \to
  \zeta$ in measure in $\Omega \times (0,t) \times D$, Vitali's
  theorem implies strong convergence in $L^1_{t,\omega,x}$ if the
  sequence $(\nabla y^{\delta}\cdot \zeta^{\delta})$ is uniformly
  integrable in $\Omega \times (0,t) \times D$. In turn, the latter is
  certainly true if $(\lvert \nabla y^{\delta}\cdot \zeta^{\delta}
  \rvert)$ is dominated by a sequence that converges strongly in
  $L^1_{t,\omega,x}$. Indeed, using the assumptions on the behavior of
  $k$ at infinity as well as the generalized Jensen inequality for
  sub-Markovian operators (see \cite{Haa07}), one has
  \[
  \pm c^2\zeta^\delta\cdot\nabla y^\delta\lesssim1+
  k(c\nabla y^\delta)+k^*(c\zeta^\delta)
  \leq1+ (I-\delta\Delta)^{-m}\left(k(c\nabla y)+k^*(c\zeta)\right),
  \]
  where the sequence on the right-hand side converges in
  $L^1_{t,\omega,x}$ as $\delta \to 0$ because, by assumption,
  $k(c\nabla y) + k^*(c\zeta) \in L^1_{t,\omega,x}$.
\end{proof}

It\^o's formula yields
\begin{align*}
  &\E \norm{u_\lambda(t)}^2 
  + 2\E\int_0^t\!\!\int_D \gamma_\lambda(\nabla u_\lambda) \cdot \nabla u_\lambda
  + 2\lambda \E \int_0^t \norm{\nabla u_\lambda}^2\\
  &\hspace{3em} = \E \norm{u_0}^2 
  + \E \int_0^t \norm{G(s)}^2_{\cL^2(H,L^2_x)}\,ds
\end{align*}
and, by Lemma \ref{lm:ito},
\[
\E\norm{u(t)}^2
+ 2 \E \int_0^t\!\!\int_ D \eta \cdot \nabla u = \E\norm{u_0}^2
+ \E \int_0^t \norm{G(s)}^2_{\cL^2(H,L^2_x)}\,ds.
\]
One then has
\begin{align*}
  &2 \limsup_{\lambda\to0} \E\int_0^T\!\!\int_D 
   \gamma_\lambda(\nabla u_\lambda(s)) \cdot \nabla u_\lambda(s)\,dx\,ds\\
  &\hspace{3em} \leq \E\norm{u_0}^2
  + \E \int_0^T \norm{G(s)}^2_{\cL^2(H,L^2_x)}\,ds
  - \liminf_{\lambda\to0} \E\norm{u_\lambda(T)}^2\\
  &\hspace{3em} \leq \E\norm{u_0}^2
  + \E \int_0^T \norm{G(s)}^2_{\cL^2(H,L^2_x)}\,ds
  - \E\norm{u(T)}^2\\
  &\hspace{3em} =\E \int_0^T\!\!\int_D \eta(s) \cdot \nabla u(s)\,dx\,ds.
\end{align*}
Since $\nabla u_\lambda \to \nabla u$ and $\gamma_\lambda(\nabla
u_\lambda) \to \eta$ weakly in $L^1_{t,\omega,x}$, this
implies that $\eta \in \gamma(\nabla u)$ a.e. in
$\Omega\times(0,T)\times D$.  We have thus proved the existence and
weak continuity statements of Proposition~\ref{prop:add}.

\medskip

In order to show that the solution is unique, we are going to prove
that \emph{any} solution depends continuously on $(u_0,G)$. Let
$(u_i,\eta_i)$, $i=1,2$, satisfy
\[
u_i(t) - \int_0^t \div \eta_i(s)\,ds = u_0 + \int_0^t G_i(s)\,ds
\]
in the sense of Definition~\ref{def:strong}, as well as the
integrability conditions (on $u$ and $\eta$) of
Theorem~\ref{thm:main}.  Setting $y := u_1-u_2$, $y_0 :=
u_{01}-u_{02}$, $\zeta:=\eta_1-\eta_2$, and $F:=G_1-G_2$, one has
\[
y(t) - \int_0^t \div\zeta(s)\,ds = y_0 + \int_0^t F(s)\,dW(s)
\]
$\P$-a.s. in $L^2(D)$ for all $t \in [0,T]$. For any process $h$, let
us use the notation $h^\alpha(t):=e^{-\alpha t} h(t)$. For any
$\alpha>0$, the integration-by-parts formula yields
\[
y^{\alpha}(t) + \int_0^t (-\div\zeta^{\alpha}(s)+\alpha y^{\alpha}(s))\,ds 
= y_0 + \int_0^t F^{\alpha}(s)\,dW(s),
\]
hence also, thanks to Lemma \ref{lm:ito},
\begin{align*}
  &\E\norm{y^{\alpha}(t)}^2
  + 2\alpha \E \int_0^t \norm{y^{\alpha}(s)}^2\,ds
  + 2 \E \int_0^t\!\!\int_D \zeta^{\alpha}(s) \cdot \nabla y^{\alpha}(s)\,dx\,ds\\
  &\hspace{3em} \leq \E \norm{y_0}^2
  + \E \int_0^t \norm{F^{\alpha}(s)}_{\cL^2(H,L^2_x)}^2\,ds,
\end{align*}
where $\zeta^\alpha \cdot \nabla y^\alpha \geq 0$ by
monotonicity. Therefore, taking the $L^\infty_t$ norm,
\[
\norm{y^\alpha}_{L^\infty_t L^2_{\omega,x}} + 
\sqrt{\alpha} \norm{y^\alpha}_{L^2_{t,\omega,x}} \lesssim
\norm{y_0}_{L^2_{\omega,x}} + \norm{F^\alpha}_{L^2_{t,\omega}\cL^2(H,L^2_x)},
\]
that is, using the notation $L^p_t(\alpha):=L^p([0,T],e^{-\alpha
  t}dt)$ for any $p \geq 0$,
\begin{equation}
  \label{eq:alfa}
\begin{split}
&\norm{u_1-u_2}_{L^\infty_t(\alpha) L^2_{\omega,x}} + 
\sqrt{\alpha} \norm{u_1-u_2}_{L^2_t(\alpha) L^2_{\omega,x}}\\
&\hspace{3em} \lesssim
\norm{u_{01}-u_{02}}_{L^2_{\omega,x}} 
+ \norm{G_1-G_2}_{L^2_t(\alpha) L^2_\omega \cL^2(H,L^2_x)}.
\end{split}
\end{equation}
Taking $\alpha = 0$ and $G_1=G_2$ immediately yields the uniqueness of
solutions (as well as Lipschitz-continuous dependence on the initial
datum). The proof of Proposition~\ref{prop:add} is thus complete.

\section{Proof of Theorem~\ref{thm:main}}
For any $v \in L^2_{t,\omega,x}$ measurable and adapted, and any
$\cF_0$-measurable random variable $u_0 \in L^2_{\omega,x}$, the
process $B(\cdot,v)$ is measurable, adapted, and belongs to
$L^2_{t,\omega} \cL^2(H,L^2_x)$, hence the equation
\[
du(t) -\div\gamma(\nabla u(t))\,dt \ni B(t,v(t))\,dW(t),
\qquad u(0)=u_0,
\]
is well-posed in the sense of Proposition~\ref{prop:add}. Moreover,
for any $v_1$, $v_2$ and $u_{01}$, $u_{02}$ satisfying the same
hypotheses on $v$ and $u_0$, respectively, \eqref{eq:alfa} yields
\begin{align*}
&\norm{u_1-u_2}_{L^\infty_t(\alpha) L^2_{\omega,x}} + 
\sqrt{\alpha} \norm{u_1-u_2}_{L^2_t(\alpha) L^2_{\omega,x}}\\
&\hspace{3em} \lesssim
\norm{u_{01}-u_{02}}_{L^2_{\omega,x}} 
+ \norm{B(\cdot,v_1)-B(\cdot,v_2)}_{L^2_t(\alpha) L^2_\omega \cL^2(H,L^2_x)}.
\end{align*}
It hence follows by the Lipschitz-continuity of $B$ that
\begin{equation}
\label{eq:alfetta}
\norm{u_1-u_2}_{L^2_t(\alpha) L^2_{\omega,x}} \lesssim
\frac{1}{\sqrt{\alpha}} \Bigl( \norm{u_{01}-u_{02}}_{L^2_{\omega,x}}
+ \norm{v_1-v_2}_{L^2_t(\alpha) L^2_{\omega,x}} \Bigr),
\end{equation}
where the implicit constant does not depend on $\alpha$. In
particular, denoting by $\Gamma$ the map $(u_0,v) \mapsto u$, one has
that $v \mapsto \Gamma(u_0,v)$ is a strict contraction of
$L^2_t(\alpha) L^2_{\omega,x}$ for $\alpha$ large enough. Therefore,
by equivalence of norms, $v \mapsto \Gamma(u_0,v)$ admits a unique
fixed point in $L^2_{t,\omega,x}$, which solves \eqref{eq:0} and
satisfies all integrability conditions. Such solution is unique as any
solution is a fixed point of $v \mapsto \Gamma(u_0,v)$.

Let us show that the solution map $u_0 \mapsto u$ is
Lipschitz-continuous: \eqref{eq:alfetta} yields, choosing $\alpha$
large enough,
\[
\norm{u_1-u_2}_{L^2_t(\alpha) L^2_{\omega,x}}
\leq N_1 \norm{u_{01}-u_{02}}_{L^2_{\omega,x}}
+ N_2 \norm{u_1-u_2}_{L^2_t(\alpha) L^2_{\omega,x}}
\]
with constants $N_1>0$ and $0 < N_2 < 1$, hence, by equivalence
of norms,
\[
\norm{u_1-u_2}_{L^2_t L^2_{\omega,x}}
\lesssim \norm{u_{01}-u_{02}}_{L^2_{\omega,x}}.
\]
This in turn implies, in view of \eqref{eq:alfa} (with $\alpha=0$) and
the Lipschitz-continuity of $B$,
\begin{align*}
  \norm{u_1-u_2}_{L^\infty_t L^2_{\omega,x}}
  &\lesssim \norm{u_{01}-u_{02}}_{L^2_{\omega,x}}
  + \norm{B(\cdot,u_1)-B(\cdot,u_2)}_{L^2_{t,\omega}\cL^2(H,L^2_x)}\\
  & \lesssim \norm{u_{01}-u_{02}}_{L^2_{\omega,x}}
  + \norm{u_1-u_2}_{L^2_t L^2_{\omega,x}}
  \lesssim \norm{u_{01}-u_{02}}_{L^2_{\omega,x}},
\end{align*}
which completes the proof.

\medskip

\noindent\emph{Remark.} A priori estimates entirely analogous to those
of Lemma~\ref{lm:aspetta}, as well as weak compactness results exactly
as in Lemma~\ref{lm:L1a}, can be proved for the regularized equation
obtained by replacing $\gamma$ with $\gamma_\lambda + \lambda\nabla$
directly in \eqref{eq:0}. It is however not immediately clear how to
pass to the limit as $\lambda \to 0$ in the stochastic integrals
appearing in such regularized equations with multiplicative noise,
i.e. to show that $B(u_\lambda) \cdot W$ converges to $B(u) \cdot W$
in a suitable sense.

\bibliographystyle{amsplain}
\bibliography{ref}

\end{document}